\documentclass[12pt]{article}
\usepackage{hyperref}
\title{Improvement of generalization of Larman-Rogers-Seidel's theorem}
\author{Cheng-Jui Yeh \\ email:
    \href{mailto:yeh.cheng.jui@gmail.com}{yeh.cheng.jui@gmail.com} 
    \and Wei-Hsuan Yu \\ email:
    \href{mailto:u690604@gmail.com}{u690604@gmail.com} }
\date{}
\usepackage{xcolor}
\usepackage{amsmath,amsthm,amssymb}
\usepackage{mathtools}
\usepackage{makeidx}
\usepackage[left=2cm, right=2cm, top=1.5cm, bottom=1.5cm]{geometry}

\theoremstyle{definition}

\newtheorem{thm}{Theorem}[section]

\newtheorem{lemma}{Lemma}[section]

\makeindex

\begin{document}
\maketitle

\begin{abstract}
    A finite set $X$ in the $d$-dimensional Euclidean space is called an $s$-distance set if the set of distances between any two distinct points of $X$ has size $s$. In 1977,
    Larman–Rogers–Seidel proved that if the cardinality of an two-distance set is large enough, then there exists an integer $k$ such that the two distances $\alpha$, $\beta$ $(\alpha<\beta)$ having the integer condition, namely,
    $\frac{\alpha^2}{\beta^2}=\frac{k-1}{k}$. In 2011, Nozaki generalized Larman–Rogers–Seidel's theorem to the case of $s$-distance sets, i.e. if the cardinality of an $s$-distance set $|X|\geqslant 2N$ with distances $\alpha_1,\alpha_2,\cdots,\alpha_s$, where $N=\binom{d+s-1}{s-1}+\binom{d+s-2}{s-2}$, then the numbers $k_i=\prod_{j=1,2,\cdots,s,\text{ }j\neq i}\frac{\alpha_{j}^{2}}{\alpha_{j}^{2}-\alpha_{i}^{2}}$ are integers. In this note, we reduce the lower bound of the requirement of integer condition of $s$-distance sets in $\mathbb{R}^d$.
    Furthermore, we can show that there are only finitely many $s$-distance sets $X$ in $\mathbb{R}^d$ with $|X|\geqslant 2\binom{d+s-1}{s-1}.$
\end{abstract}

\section{Introduction}

\hspace{0.6cm}Let $\mathbb{R}^d$ be the $d$-dimensional Euclidean space, and $|*|$ denote the cardinality. A finite set $X$ in $\mathbb{R}^d$ is said to be an $s$-distance set if there are exactly $s$ distinct distances between the points in $X$, i.e. $|A(X)|=s$, where $A(X)=\{\|x-y\|:x,y\in X\}$ be the set of the distances on $X$. If the size of an $s$-distance set is large enough, we are curious about whether there are any relations between the $s$ distinct distances. In 1977, Larman, Rogers, and Seidel (LRS) \cite{LRS} discovered the following theorem:
\begin{thm}\label{LRS_thm}\cite{LRS}
Let $X$ be a set in $\mathbb{R}^d$ that only realises the two distances $\alpha$ and $\beta$ with $\alpha<\beta$. If $|X|>2d+3$, then $\frac{\alpha^2}{\beta^2}=\frac{k-1}{k}$, with $k\geqslant 2$ a positive integer, and $k\leqslant \frac{1}{2}+\sqrt{\frac{1}{2}d}$.
\end{thm}
In 1981, Neumaier \cite{Neumaier} reduced the requirement of $|X|$ from $2d+3$ in Theorem \ref{LRS_thm} to $2d+1$ and found  counterexamples with $2d+1$ points in $\mathbb{R}^d$. The counterexamples are the spherical embedding of conference graphs. In 2011, Nozaki \cite{Nozaki} generalized the LRS Theorem to the case of $s$-distance sets, and found that the ratios of the $s$ distances also have integer conditions if the cardinality of $s$-distance sets is large enough:
\begin{thm}\label{Nozaki_thm}\cite{Nozaki}%(Nozaki, 2011)\\
\text{ }Let $X$ be an $s$-distance set in $\mathbb{R}^{d}$ with $s\geqslant 2$ and $A(X)=\{\alpha_{1},\cdots, \alpha_{s}\}$. Let $N=\binom{d+s-1}{s-1}+\binom{d+s-2}{s-2}$. If $|X|\geqslant 2N$, then \[\prod_{j=1,2,\cdots,s,\text{ }j\neq i}\frac{\alpha_{j}^{2}}{\alpha_{j}^{2}-\alpha_{i}^{2}}\] is an integer $k_{i}$ for each $i=1,2,\cdots, s$. Furthermore, $|k_{i}|\leqslant \lfloor\frac{1}{2}+\sqrt{\frac{N^{2}}{2N-2}+\frac{1}{4}}\rfloor$.
\end{thm}
However, in the case of $s=2$, the condition in Theorem \ref{Nozaki_thm} gives $|X|>2d+3$, which is larger than the condition $|X|>2d+1$ improved by Neumaier. In this note, we prove that the latter term of the bound $|X|\geqslant 2\binom{d+s-1}{s-1}+2\binom{d+s-2}{s-2}$ can be removed, and hence the theorem becomes as follows:
\begin{thm}\label{our_main_thm}
Let $X$ be an s-distance set in $\mathbb{R}^{d}$ with $s\geqslant 2$ and $A(X)=\{\alpha_{1},\cdots, \alpha_{s}\}$. Let $N=\binom{d+s-1}{s-1}$. If $|X|\geqslant 2N$, then \[\prod_{j=1,2,\cdots,s,\text{ }j\neq i}\frac{\alpha_{j}^{2}}{\alpha_{j}^{2}-\alpha_{i}^{2}}\] is an integer $k_{i}$ for each $i=1,2,\cdots, s$. Furthermore, $|k_{i}|\leqslant \lfloor\frac{1}{2}+\sqrt{\frac{N^{2}}{2N-2}+\frac{1}{4}}\rfloor$.
\end{thm}
In particular, if $s=2$, the condition gives $|X|>2d+1$, which is same as the best situation in the case of two-distance sets.

In section 2, we will recall some lemmas that will be used later. In section 3, we will prove our main result Theorem \ref{our_main_thm}. In section 4, we will reduce the requirement of the size of $s$-distance sets to guarantee that there are only finitely many $s$-distance sets $X$ in $\mathbb{R}^d$ with $|X|\geqslant 2\binom{d+s-1}{s-1}$, due to our improvement of Theorem \ref{Nozaki_thm}.

%%%%%%%%%%%%%%%%%%%%%%%%%%%%%%%%%%%%%%%%%%%%%%%%%%%%%%%%

\section{Preliminaries}

\hspace{0.6cm}Here, we denote $P_{\ell}(\mathbb{R}^{d})$ for the space of all the polynomials with $d$ variables of degree $\leq\ell$. Let $x_{1},\cdots, x_{d}$ be independent variables, and set $x_{0}=x_{1}^{2}+\cdots +x_{d}^{2}$. We denote $W_{\ell}(\mathbb{R}^{d})$ for the linear space spanned by monomials $x_{0}^{\lambda_{0}}x_{1}^{\lambda_{1}}\cdots x_{d}^{\lambda_{d}}$ with $\lambda_{0}+\lambda_{1}+\cdots \lambda_{d}\leqslant \ell$ and $\lambda_{i}\geqslant 0$. 
First, we recall that Theorem \ref{Nozaki_thm} in \cite{Nozaki}  is proved by using the following lemma:
\begin{lemma}\label{main_lemma_Nozaki}\cite{Nozaki}
Let $X$ be a finite subset of $\Omega\subset\mathbb{R}^{d}$, and $\mathcal{P}(\mathbb{R}^{d})$ a linear subspace of $P_{\ell}(\mathbb{R}^{d})$.
Let $N=dim(\mathcal{P}(\Omega))$. Suppose that there exists $F_{x}\in \mathcal{P}(\Omega)$ for each $x\in X$ such that $F_{x}(x)=k$ where k is a constant, $F_{x}(y)$ are 0 or 1 for all $y\in X$ with $y\neq x$, and $F_{x}(y)=F_{y}(x)$ for all $x,y\in X$. If $|X|\geq 2N$, then $k$ is an integer and $|k|\leqslant\lfloor\frac{1}{2}+\sqrt{\frac{N^{2}}{2N-2}+\frac{1}{4}}\rfloor$.
\end{lemma}
From Lemma \ref{main_lemma_Nozaki}, we can see that the  lower bound \[|X|\geqslant 2\binom{d+s-1}{s-1}+2\binom{d+s-2}{s-2}\]
in Theorem \ref{Nozaki_thm} is coming from the dimension of the function space where the set of functions
\begin{equation}\label{set_of_funcs_Nozaki}
    \{\text{ }F_{y}(x)=\prod_{j=1,2,\cdots,s,\text{ }j\neq i}\frac{\alpha_{j}^{2}-\|y-x\|^{2}}{\alpha_{j}^{2}-\alpha_{i}^{2}}\text{ }:\text{ }y\in X\text{ }\}\text{ }\subset W_{s-1}(\mathbb{R}^d)
\end{equation}
lives in. In fact, a similar set of functions 
\begin{equation}\label{set_of_funcs}
    \{\text{ }G_{y}(x)=\prod_{j=1,2,\cdots,s}\frac{\alpha_{j}^{2}-\|y-x\|^{2}}{\alpha_{j}^{2}}\text{ }:\text{ }y\in X\text{ }\}\text{ }\subset W_{s}(\mathbb{R}^d)
\end{equation}
had been considered in \cite{Bannai_2} by Bannai, Bannai, and Stanton in 1983 for getting new upper bounds of the cardinality of $s$-distance sets in $\mathbb{R}^d$. They found that $\{G_{y}(x):y\in X\}$ and $\{x_{1}^{\lambda_{1}}x_{2}^{\lambda_{2}}\cdots x_{d}^{\lambda_{d}}:0\leqslant\lambda_{1}+\lambda_{2}+\cdots+\lambda_{d}\leqslant s-1\}$ are linearly independent sets in $W_{s}(\mathbb{R}^d)$, and therefore improved the upper bounds of $s$-distance sets to \[|X|\leqslant\binom{d+s}{s}.\] 
In section 3, we show that the ideal in \cite{Bannai_2} can be used to prove that $\{F_{y}(x):y\in X\}$ and $\{x_{1}^{\lambda_{1}}x_{2}^{\lambda_{2}}\cdots x_{d}^{\lambda_{d}}:0\leqslant\lambda_{1}+\lambda_{2}+\cdots+\lambda_{d}\leqslant s-2\}$ are linearly independent sets in $W_{s-1}(\mathbb{R}^d)$, and thus the lower bound in Theorem \ref{Nozaki_thm} can be improved.

Before going into the proof of Theorem \ref{our_main_thm}, we recall some Lemmas in \cite{Bannai_1}, \cite{Bannai_2} which will be used later:

\begin{lemma}\label{dim_W}\cite{Bannai_1}
    \[\text{dim}(W_{\ell}(\mathbb{R}^{d}))\leqslant \binom{d+\ell}{\ell}+\binom{d+\ell-1}{\ell-1}\]
\end{lemma}

\begin{lemma}\label{change_basis}\cite{Bannai_2}
    Fix $s'\in\mathbb{N}$. Let $x_1,\cdots,x_d$ be independent variables, and denote $\partial_i=\frac{\partial}{\partial x_i}$. Let $2\leqslant\ell\leqslant s'+2$, then we have
    \begin{align*}
        &\text{the space spanned by } \{\partial_1^{b_1}\cdots\partial_d^{b_d}(x_1^2+\cdots+x_d^2)^{s'}\text{ : } b_1+\cdots+b_d=2s'-\ell+2\}\\
        =\text{ }&\text{the space spanned by } \{x_1^{\alpha_1}x_2^{\alpha_2}\cdots x_d^{\alpha_d}\text{ : }\alpha_1+\cdots+\alpha_d=\ell-2\}
    \end{align*}
\end{lemma}
The following lemma is a modification of lemma in Bannai-Bannai-Stanton \cite{Bannai_2}. Since some of the parameters are different to the original theorem, we give it a prove.
\begin{lemma}\label{m_i_lemma}
    For $i=1,\cdots, N$, let $m_i\in\mathbb{R}$ and $y^{(i)}=(y_1^{(i)},\cdots,y_d^{(i)})\in\mathbb{R}^d$. For fixed integers $2\leqslant\ell\leqslant s+1$, suppose \[\sum_{i=1}^{N}m_{i}\|x-y^{(i)}\|^{2(s-1)}\] is a polynomial in $x$ of degree $\leqslant 2s-\ell-1$, then \[\sum_{i=1}^{N}m_{i}(y_1^{(i)})^{\alpha_1}\cdots(y_d^{(i)})^{\alpha_d}=0, \hspace{0.7cm}\forall\text{ } 0\leqslant\alpha_1+\cdots+\alpha_d\leqslant\ell-2.\]
\end{lemma}
\begin{proof}
    Note that $\sum_{i=1}^{N}m_{i}\|x-y^{(i)}\|^{2(s-1)}$ is a polynomial of degree $\leqslant 2s-\ell-1$ in $x$ if and only if
    \begin{equation}\label{(0)}
        \partial_1^{b_1}\partial_2^{b_2}\cdots\partial_d^{b_d}(\sum_{i=1}^{N}m_{i}\|x-y^{(i)}\|^{2(s-1)})=0,\hspace{0.7cm}\forall\text{ }2s-\ell\leqslant b_1+\cdots+b_d\leqslant 2(s-1).
    \end{equation}
    By lemma \ref{change_basis}, (choose $s'=s-1$) 
    \[(x_1-y_1^{(i)})^{\alpha_1}\cdots(x_d-y_d^{(i)})^{\alpha_d}=\sum_{b_1+\cdots+b_d=2s'-\ell+2=2s-\ell} C_{b_{1}\cdots b_{d}}^{\alpha_{1}\cdots\alpha_{d}}\text{ }\partial_1^{b_1}\cdots\partial_d^{b_d}\text{ }[(x_1-y_1^{(i)})^2+\cdots+(x_d-y_d^{(i)})^2]^{s-1}\] 
    for some real $C_{b_{1}\cdots b_{d}}^{\alpha_{1}\cdots\alpha_{d}}$. So, as polynomials in $x$, we have 
    \begin{align*}
        &\sum_{i=1}^{N}\text{ }m_{i}(x_1-y_1^{(i)})^{\alpha_1}\cdots(x_d-y_d^{(i)})^{\alpha_d}\\
        &=\sum_{i=1}^{N}\text{ }m_{i}\sum_{b_1+\cdots+b_d=2s-\ell} C_{b_{1}\cdots b_{d}}^{\alpha_{1}\cdots\alpha_{d}}\text{ }\partial_1^{b_1}\cdots\partial_d^{b_d}\text{ }[(x_1-y_1^{(i)})^2+\cdots+(x_d-y_d^{(i)})^2]^{s-1}\\
        &=\sum_{b_1+\cdots+b_d=2s-\ell} C_{b_{1}\cdots b_{d}}^{\alpha_{1}\cdots\alpha_{d}}\sum_{i=1}^{N}\text{ }m_{i}\text{ }\partial_1^{b_1}\cdots\partial_d^{b_d}\text{ }[(x_1-y_1^{(i)})^2+\cdots+(x_d-y_d^{(i)})^2]^{{s-1}}\\
        &=0.\text{ }\text{ }(\text{by } (\ref{(0)}))
    \end{align*}
    By putting $x_1=\cdots=x_d=0$, we have done.
\end{proof}

%%%%%%%%%%%%%%%%%%%%%%%%%%%%%%%%%%%%%%%%%%%%%%%%%%%%%%%%%%%%%%%%%%
\section{Improvement of Generalization of LRS Theorem}
\hspace{0.6cm}Now, we can prove Theorem \ref{our_main_thm}:

\begin{proof}[Proof of Theorem \ref{our_main_thm}]
Fix $i\in\{1,2,\cdots,s\}$. For each $y\in X$, we define the polynomial
\begin{align*}
    F_{y}(x)&=\prod_{j=1,2,\cdots,s,\text{ }j\neq i}\frac{\alpha_{j}^{2}-\|x-y\|^{2}}{\alpha_{j}^{2}-\alpha_{i}^{2}},
\end{align*}
then $F_{y}\in W_{s-1}(\mathbb{R}^{d})$ for each $y\in X$. In order to prove Theorem \ref{our_main_thm}, we only need to show that the two sets 
\begin{equation}\label{indep}
    \text{ }\text{ }\{F_{y}(x):y\in X\}\text{ }\text{ and }\text{ }\{x_{1}^{\lambda_{1}}x_{2}^{\lambda_{2}}\cdots x_{d}^{\lambda_{d}}:0\leqslant\lambda_{1}+\lambda_{2}+\cdots+\lambda_{d}\leqslant s-2\}
\end{equation}
are linearly independent functions on $\mathbb{R}^{d}$, because
\[\text{dim}(\text{span}(\{x_{1}^{\lambda_{1}}x_{2}^{\lambda_{2}}\cdots x_{d}^{\lambda_{d}}:0\leqslant\lambda_{1}+\lambda_{2}+\cdots+\lambda_{d}\leqslant s-2\}))=\binom{d+s-2}{s-2}\] 
and Lemma \ref{dim_W} gives
\[\text{dim}(W_{s-1}(\mathbb{R}^{d}))\leqslant\binom{d+s-1}{s-1}+\binom{d+s-2}{s-2}.\]
By the independence of (\ref{indep}), we have \[\text{dim}(\text{span}\{F_{y}(x):y\in X\})\leqslant\binom{d+s-1}{s-1}+\binom{d+s-2}{s-2}-\binom{d+s-2}{s-2}=\binom{d+s-1}{s-1}.\] By the following two conditions:
\begin{enumerate}
    \item[1.] $\text{span}(\{F_{y}(x):y\in X\})$ is a linear subspace of $P_{s-1}(\mathbb{R}^{d})$ with dimension $\leqslant\binom{d+s-1}{s-1}$.
    \item[2.] $F_{y}(y)=\prod_{j\neq i}\frac{\alpha_{j}^{2}}{\alpha_{j}^{2}-\alpha_{i}^{2}}$ is a constant for all $y\in X$, $F_{y}(x)=1$ if $d(x,y)=\alpha_{i}$, $F_{y}(x)=0$ if $d(x,y)\neq\alpha_{i}$, and $F_{y}(x)=F_{x}(y)$ for all $x,y\in X$.
\end{enumerate}
in conjunction of Lemma \ref{main_lemma_Nozaki}, the theorem follows. 
Now, we prove the independence of 
(\ref{indep}). Suppose
\begin{equation}\label{suppose_sum=0}
    \sum_{y\in X}C_yF_y(x)+\sum_{0\leqslant \lambda_1+\lambda_2+\cdots+\lambda_d\leqslant s-2}C_{\lambda_{1}\lambda_{2}\cdots\lambda_{d}}\text{ }x_{1}^{\lambda_1}x_{2}^{\lambda_2}\cdots x_{d}^{\lambda_d}=0
\end{equation} for some $C_y, C_{\lambda_{1}\lambda_{2}\cdots\lambda_{d}}=C_{\lambda}\in\mathbb{R}^d$, with $y\in X, 0\leqslant \lambda_1+\lambda_2+\cdots+\lambda_d=\lambda\leqslant s-2$.
Choosing $x=u\in X$ in (\ref{suppose_sum=0}), we get \[(\prod_{j\neq i}\frac{\alpha_{j}^{2}}{\alpha_{j}^{2}-\alpha_{i}^{2}})C_u+\sum_{\lambda}C_{\lambda}u^{\lambda}=0\] By multiplying $C_u$ and summing over $u\in X$, we get 
\begin{equation}\label{sum_C_u}
    (\prod_{j\neq i}\frac{\alpha_{j}^{2}}{\alpha_{j}^{2}-\alpha_{i}^{2}})\sum_{y\in X}C_y^2+\sum_{\lambda}C_{\lambda}\sum_{y\in X}C_{y}\text{ }y_{1}^{\lambda_1}y_{2}^{\lambda_2}\cdots y_{d}^{\lambda_d}=0.
\end{equation}
It is enough to show 
\begin{equation}\label{claim_2}
    \sum_{y\in X}C_{y}\text{ }y_{1}^{\lambda_1}y_{2}^{\lambda_2}\cdots y_{d}^{\lambda_d}=0
\end{equation} for any $0\leqslant \lambda_1+\lambda_2+\cdots+\lambda_d\leqslant s-2$, since this implies $\sum_{y\in X}C_y^2=0$ in (\ref{sum_C_u}), and hence $C_y=0$ for all $y\in X$.

Now, we prove (\ref{claim_2}) by using induction on $\lambda:=\lambda_1+\lambda_2+\cdots+\lambda_d$. For the basic case  $\lambda=0$, it is clear that \[\sum_{y\in X}C_y=0\] by comparing the term of degree $2s-2$ in (\ref{suppose_sum=0}). Assume that (\ref{claim_2}) holds for $0\leqslant\lambda\leqslant\ell-3$, and show that (\ref{claim_2}) holds for $0\leqslant\lambda\leqslant\ell-2$ if $3\leqslant\ell\leqslant s$.
Expand
\begin{equation}\label{expansion}
    \sum_{y\in X}C_{y}F_{y}=\sum_{y\in X}C_{y}\sum_{t=1}^{s}A_{t}\|x-y\|^{2(s-t)}=\sum_{t=1}^{s}A_{t}\sum_{y\in X}C_{y}\|x-y\|^{2(s-t)}
\end{equation} for some $A_t\in\mathbb{R}$. By induction hypothesis, the $y^\lambda$ terms with degree $0\leqslant\lambda\leqslant\ell-3$ vanish in $\sum_{y\in X}C_{y}\|x-y\|^{2(s-t)}$, so $\sum_{y\in X}C_{y}\|x-y\|^{2(s-t)}$ as polynomial in $x$ has degree at most $2(s-t)-(\ell-2)$. 
Note that the only term in (\ref{expansion}) which allows degree $2s-\ell$ in $x$ is the term $t=1$, so \[\sum_{y\in X}C_{y}\|x-y\|^{2(s-1)}\] is a polynomial in $x$ of degree $\leqslant2s-\ell$. 
Since there are no terms in $\sum_{y\in X}C_{y}F_{y}$ with degree $2s-\ell$ (by (\ref{suppose_sum=0})), so $\sum_{y\in X}C_{y}\|x-y\|^{2(s-1)}$ has degree $\leqslant 2s-\ell-1$.
Using Lemma \ref{m_i_lemma} (with $N=|X|$, $m_i=C_y$), we have (\ref{claim_2}) holding for $0\leqslant\lambda\leqslant\ell-2$ if $3\leqslant\ell\leqslant s$. Thus, by induction, we have shown (\ref{claim_2}). This completes the proof of Theorem \ref{our_main_thm}.
\end{proof}

%%%%%%%%%%%%%%%%%%%%%%%%%%%%%%%%%%%%%%%%%%%%%%%%%%%%%%%%%%%%%%%%%%%
\section{Finitely Many $s$-distance Sets}

\hspace{0.6cm}In 1966, Einhorn and Schoenberg \cite{Schoenberg} proved that there are only finitely many two-distance sets $X$ in $\mathbb{R}^d$ with $|X|\geq d+2$. In $\cite{Nozaki}$, Nozaki also generalized this theorem from the situation of two-distance sets to $s$-distance sets, which stated as below:
\begin{thm}\label{finite_s}(\cite{Nozaki})
    There are finitely many $s$-distance sets $X$ in $\mathbb{R}^d$ with \[|X|\geqslant 2\binom{d+s-1}{s-1}+2\binom{d+s-2}{s-2}.\]
\end{thm}
This theorem is proved by using Theorem \ref{Nozaki_thm} and the following lemma in \cite{Nozaki}.

\begin{lemma}\label{lemma_finite_nozaki}(\cite{Nozaki})
    Let $X$ be an $s$-distance set with $|X|\geqslant 2\binom{d+s-1}{s-1}+2\binom{d+s-2}{s-2}$ and $A(X)=\{\alpha_1,\alpha_2,\cdots,\alpha_s=1\}$. Suppose $k_i$ are the ratios in Theorem \ref{Nozaki_thm}. Then the distances $\alpha_i$ are uniquely determined from given integers $k_i$.
\end{lemma}
The bound $|X|\geqslant 2\binom{d+s-1}{s-1}+2\binom{d+s-2}{s-2}$ in Lemma \ref{lemma_finite_nozaki} is actually coming from the bound in Theorem \ref{Nozaki_thm}. Since this bound have been improved by our Theorem \ref{our_main_thm}, we can also reduce the requirement of Theorem \ref{finite_s} to the bound $|X|\geqslant 2\binom{d+s-1}{s-1}$ immediately. Therefore, we have the following theorem:
\begin{thm}\label{our_second_result}
    There are finitely many $s$-distance sets $X$ in $\mathbb{R}^d$ with \[|X|\geqslant 2\binom{d+s-1}{s-1}.\]
\end{thm}
However, for $s=2$, Theorem \ref{our_second_result} gives $|X|\geqslant 2\binom{d+2-1}{2-1}=2d+2$, which is bigger than the bound $|X|\geqslant d+2$ given by Einhorn and Schoenberg. Hence, the bound in Theorem \ref{our_second_result} might be improved. 

\section{Discussion}

\hspace{0.6cm}If $s=2$, Theorem \ref{our_main_thm} proves that any two-distance set $X$ in $\mathbb{R}^d$ with $|X|\geqslant 2\binom{d+2-1}{2-1}=2d+2$ has the integer condition. This is the same as the result of Neumaier's improvement of LRS theorem in 1981. Furthermore, if $X$ is a two-distance set in $\mathbb{R}^d$ with $|X|=2d+1$, Neumaier provided the counterexamples from the spherical embedding of conference graphs. Therefore, the bound of $|X|$ in Theorem \ref{our_main_thm} cannot be improved for any $s$-distance set in $\mathbb{R}^d$ in general. We are curious about whether the bounds can be improved or not for $s\geqslant 3$.

We think that those $s$-distance sets in $\mathbb{R}^d$ without integer condition can be obtained from the coherent configurations which contains a spherical embedding of conference graphs. However, such example may not have the size larger enough to attain the lower bounds $2\binom{d+s-1}{s-1}$ in Theorem \ref{our_main_thm}.
It is interesting to ask if there exists $s$-distance sets with size $2\binom{d+s-1}{s-1}-1$ in $\mathbb{R}^d$ without integer condition for some $s\geqslant 3$. \\

\textbf{Acknowledgement} We sincerely appreciate Eiichi Bannai for the useful discussion of this manuscript. We also thank NCTS supporting this project as the NCTS undergraduate research program. 

%%%%%%%%%%%%%%%%%%%%%%%%%%%%%%%%%%%%%%%%%%%%%%%%%%%%%%%%%%
%\medskip

\end{document}